\documentclass[11pt, letterpaper]{article}
\usepackage{amsfonts}
\usepackage{amssymb}
\usepackage{latexsym}  % latex2e symbol package
\usepackage{pifont}  % so we can use ZapfDingbats
\usepackage{mathscinet}  % so \Dbar, as in Kurepa, works.
\usepackage{amsmath}   % so we can use \tag , etc.

\let\SavedRightarrow=\Rightarrow
\usepackage{marvosym}
\let\Rightarrow=\SavedRightarrow

\newtheorem{theorem}{Theorem}[section]
\newtheorem{definition}[theorem]{Definition}
\newtheorem{lemma}[theorem]{Lemma}

\newcommand\ZFC{\mathit{ZFC}}
\newcommand\MA{\mathit{MA}}
\newcommand\CH{\mathit{CH}}
\newcommand\LL{\mathcal{L}}
\newcommand\WW{\mathcal{W}}
\newcommand\height{\mathrm{height}}  % height

\newcommand\ran{\mathrm{ran}}
\newcommand\diam{\mathrm{diam}}  % diameter

\newcommand\RRR{{\mathbb R}} 
\newcommand\QQQ{{\mathbb Q}}

\newcommand\res{\mathord {\upharpoonright}}  % less space around it

\newcommand\cat{^{\mathord{\frown}}} % frown, used for concatenation

\newcommand\down{{\fam0 \downarrow}}
\newcommand\up{{\fam0 \uparrow}}

\newcommand\iv{^{-1}}   % inverse

\newcommand\tro{\sqsubset}  % squared-off less than symbol -- used as tree ord.
\newcommand\eop{\raisebox{-2pt}{\mbox{\Huge \Smiley}}}

\newenvironment{proof}{{\bf Proof.}}{\eop\medskip}
\newenvironment{proofof}[1]{\medskip \textbf{Proof of #1.}}{\eop\medskip}

\newenvironment{itemizz}{\begin{itemize}\setlength{\itemsep}{-1mm}}
{\end{itemize}}

\begin{document}

\title{Continuous Maps on Aronszajn Trees%
\footnote{
2010 Mathematics Subject Classification:
Primary  03E35, 54F05.
Key Words and Phrases: 
special Aronszajn tree, diamond.
}}

\author{Kenneth Kunen\footnote{University of Wisconsin,  %
Madison, WI  53706 \ \ kunen@math.wisc.edu},
Jean A. Larson\footnote{University of Florida,  %
Gainesville, FL 32611 \ \ jal@ufl.edu},
and Juris Stepr\=ans\footnote{York University,  %
Toronto, Ontario M3J 1P3  \ \ steprans@yorku.ca}  }

\maketitle

\begin{abstract}
Assuming $\diamondsuit$:  Whenever $B$ is a totally imperfect
set of real numbers, there is
special Aronszajn tree with no continuous order preserving map into $B$.
\end{abstract}

\section{Introduction}
\label{sec-intro}
We use the following notation:  If $\tro$ is a relation on $T$ and $x \in T$,
then $x\up$ denotes $\{y \in T : x \tro y\}$ and
$x\down$ denotes $\{y \in T : y \tro x\}$.  Then a 
\emph{tree} is a set $T$ with a strict partial order $\tro$
such that each $x\down$ is well-ordered by $\tro$.
In a tree $T$, $\height(x)$ is the order type of $x\down$ and
$\LL_\alpha = \LL_\alpha(T) = \{x \in T : \height(x) = \alpha\}$.
$T$ is an \emph{$\omega_1$--tree} iff
$|T| = \aleph_1$, each $\LL_\alpha(T)$ is countable,
and $\LL_{\omega_1}(T) = \emptyset$. 
An \emph{Aronszajn tree} is an $\omega_1$--tree $T$ with no uncountable chains;
then, $T$ is \emph{special} iff $T$ is a countable union of antichains.

We give a tree $T$ its natural \emph{tree topology}, in which
$U \subseteq T$ is open iff for all  $y \in U$ with $\height(y)$ a limit
ordinal, there is an $x \tro y$ such that $x\up \cap y\down \subseteq U$.
Then the elements whose heights are successor ordinals or $0$ are
isolated points.  Note that $T$ need not be Hausdorff,
although any tree that we construct explicitly will be Hausdorff
(equivalently, $y\down = z\down \to y = z$).

Let $T$ be an $\omega_1$--tree.
A map $\varphi : T \to \RRR$ is called \emph{order preserving}
iff $x \tro y \to \varphi(x) < \varphi(y)$ for all $x,y \in T$.
The existence of such a $\varphi$
clearly implies that $T$
is Aronszajn, but not necessarily special; there is a counter-example
\cite{Devlin} under $\diamondsuit$.
However, it is easy to see (first noted by Kurepa \cite{Kurepa})
that $T$ is special iff there is an order preserving $\varphi: T \to \QQQ$.

Let $T$ be an Aronszajn tree.  If there is an order preserving
$\varphi : T \to \RRR$, then there is also a \emph{continuous}
order preserving $\psi : T \to \RRR$, where
$\psi(y) = \varphi(y)$ unless $\height(y)$ is a limit ordinal,
in which case $\psi(y) = \sup\{\varphi(x) : x \tro y\}$.
If we assume $\MA(\aleph_1)$, then every Aronszajn tree is special,
as Baumgartner \cite{Baumgartner} proved by forcing with
finite order preserving maps into $\QQQ$.  Note that this same
forcing also produces a \emph{continuous}
order preserving $\psi : T \to \QQQ$.
We show here that this cannot be done in $\ZFC$, since assuming
$\diamondsuit$, there is an Aronszajn tree $T$ with an order
preserving map into $\QQQ$ (so $T$ is special), but
no continuous order preserving $\psi : T \to \QQQ$.\footnote{%
A continuous order preserving map $\psi$ from an Aronszajn tree $T$
into the rationals is a nice thing to have.  Todor\v{c}evi\'{c}
\cite[Remark 4.3.(d) on page 429]{Todorcevic} proved that a combination of such
a map with his osc map can be used to color the $2$-element
chains of $T$ with countably many colors so that every chain
of order type $\omega^\omega$ receives all the colors.}

This last result can be generalized somewhat.
First, we can replace ``order preserving'' by the
weaker requirement that each $\psi\iv\{q\}$ is discrete
in the tree topology; observe that when $\psi$ is order preserving,
each  $\psi\iv\{q\}$ is an antichain, and hence closed and discrete.
Then, we can replace $\QQQ$ by any metric space
which has no Cantor subsets (that is, subsets homeomorphic
to $2^\omega$):

\begin{theorem}
\label{thm-special-diamond}
Assume $\diamondsuit$, and fix a metric space $B$
with no Cantor subsets such that $|B| \le \aleph_1$.  Then
there is a special Aronszajn tree $T$
which has no continuous map $\psi : T \to B$ such that each
$\psi\iv \{b\}$ is discrete.
\end{theorem}

By $\CH$ (which follows from $\diamondsuit$), $|B| \le \aleph_1$
holds whenever $B$ is separable, as well as when $B$ has 
a dense subset of size $\aleph_1$.

Observe that if $T$ is special and
$B \subseteq \RRR$ does have a Cantor subset $F$,
then there must be a continuous order
preserving $\psi : T \to B$.  Just let $D \subseteq F$ be
countable and order-isomorphic to $\QQQ$, let
$\varphi : T \to D$ be order preserving, and then construct
a continuous $\psi : T \to F$ as described above.

In Theorem \ref{thm-special-diamond}, $T$ depends on $B$.
There is no one tree which works for all $B$ by the following,
which holds in $\ZFC$ (although it is trivial unless $\CH$ is true):

\begin{theorem}
\label{thm-build-map}
Let $T$ be any special Aronszajn tree.  Then
there is a $B \subseteq \RRR$ with no Cantor subsets
and a continuous order preserving map $\psi : T \to B$ such that
for all $x,y \in T$, $\psi(x) \ne \psi(y)$ unless $x\down = y\down$.
\end{theorem}

So, $\psi$ is actually 1-1 if $T$ is Hausdorff.
Theorem \ref{thm-special-diamond} is proved in Section \ref{sec-kill},
and Theorem \ref{thm-build-map} is proved in Section \ref{sec-const}.

By Theorem \ref{thm-build-map}, the ``$|B| \le \aleph_1$'' cannot be removed
in Theorem \ref{thm-special-diamond}, since $B$ could be the direct sum
of all totally imperfect subspaces of $\RRR$.

\section{Killing Continuous Maps}
\label{sec-kill}
Throughout, $T$ always denotes an $\omega_1$--tree and $B$
denotes a metric space.
We begin with some remarks on pruning open $U \subseteq T$.
In the special case when $U$ is a subtree (that is,
$x\down \subseteq U$ for all $x \in U$), the pruning
reduces to the standard procedure of removing all $x \in U$
with $x\up \cap U$ countable.  For a general $U$, we replace
``countable'' by ``non-stationary'' (which is the same when $U$ is a subtree).

\begin{definition}
For $U \subseteq T$: $U$ is \emph{stationary} iff $\{\height(x) : x \in U\}$
is stationary, and $U^p$ is the set of all $x \in U$
such that $x\up \cap U$ is stationary.
\end{definition}

Clearly $U^p \subseteq U$.  If $U$ is open then
$U^p$ is open, since $x \in U^p \to x\down \cap U \subseteq U^p$.

\begin{lemma}
If $U \subseteq T$ is open, then $(U^p)^p = U^p$.
\end{lemma}
\begin{proof}
Fix $a \in U^p$; so $a\up \cap U$ is stationary.  We need to show:
$\{x \in a\up \cap U : x\up \cap U \text{ is stationary}\} 
\ \text{ is stationary}$.
So, we fix a club $C \subseteq \omega_1$, and we shall  find an $x$ such
that $\height(x) \in C$ and 
$a \tro x$ and $x \in U$ and $x\up \cap U$ is stationary.

Since $a \in U^p$, fix a stationary $S$ such that for all
$\beta \in S$:  $a\up \cap U \cap \LL_\beta(T) \ne \emptyset$
and $\beta$ is a limit point of $C$.
For each $\beta \in S$:  Choose $y_\beta \in a\up \cap U \cap \LL_\beta(T)$;
then, since $U$ is open, choose $x_\beta \tro y_\beta$ such that
$x_\beta \in  a\up \cap U$ and $\height(x_\beta) \in C$.

By the Pressing Down Lemma, fix $x$ and a stationary $S' \subseteq S$
such that $x_\beta = x$ for all $x \in S'$.  Then
$x\up \cap U$ is stationary (since it contains $\{y_\beta : \beta \in S'\}$)
and $\height(x) \in C$ and $a \tro x$ and $x \in U$.
\end{proof}

\begin{lemma}
\label{lemma-stat-minus-ac}
If $A \subseteq T$ is discrete in the tree topology
and $U$ is a stationary open 
set, then the set $S := \{\alpha : U \cap \LL_\alpha \ne \emptyset \,\wedge\,
 U \cap \LL_\alpha  \subseteq A\}$ is non-stationary.
Hence, $U \backslash A$ is stationary.
\end{lemma}
\begin{proof}
In fact, $S$ is discrete in the ordinal (= tree) topology on $\omega_1$.
To see this, suppose that $\alpha \in S$ is a limit ordinal.
Then fix $y \in U \cap \LL_\alpha$.
Note that $y \in A$ since $ U \cap \LL_\alpha  \subseteq A$.
Since $U$ is open and $A$ is discrete, we may
fix  $x \tro y$ such that $x \up \cap y \down \subseteq U$
and $x \up \cap y \down \cap A = \emptyset$.  Let $\xi = \height(x)$. 
Then $\xi < \alpha$, and $S$ contains no ordinals between $\xi$ and $\alpha$.
\end{proof}

The next lemma has a much simpler proof when $B$ is separable 
(then, each $\WW_n$ can be a singleton).
For $b \in B$ and $\varepsilon > 0$, let
$N_\varepsilon(b) = \{z \in B : d(b,z) < \varepsilon\}$
(where $d$ is the metric on $B$).

\begin{lemma}
\label{lemma-parac}
Suppose that $U \subseteq T$ is a stationary open set,
$B$ is any metric space, and
$\psi:  U \to B$ is continuous, with each $\psi\iv\{b\}$
discrete.
Then there are infinitely many $b \in B$ such that
$\psi\iv(N_\varepsilon(b))$ is stationary for all $\varepsilon > 0$.
\end{lemma}
\begin{proof}
Since each $U \setminus \psi\iv\{b\}$ is also stationary open by Lemma
\ref{lemma-stat-minus-ac}, it is sufficient to prove that there is one such $b$.
If there are no such $b$, then $B$ is covered by the open sets $W$ such
that $\psi\iv(W)$ is non-stationary.  By paracompactness of $B$,
this cover has a $\sigma$--discrete open refinement, $\{\WW_n : n \in \omega\}$.
So, each $\WW_n$ is a discrete (and hence disjoint) family of open sets $W$
such that $\psi\iv(W)$ is non-stationary,
and $B = \bigcup_{n\in\omega} (\bigcup \WW_n)$.

Fix $n$ such that $\psi\iv(\bigcup \WW_n)$ is stationary.
We may assume that $| \WW_n | \ge \aleph_1$,
since $| \WW_n | \le \aleph_0$ yields an obvious contradiction.
Also, we may assume that $|B| \le \aleph_1$
(replacing $B$ by $\psi(U) )$, so that $| \WW_n | = \aleph_1$.
Let $\WW_n = \{W_\xi : \xi < \omega_1\}$.

For each $\xi$, let $C_\xi$ be a club disjoint from
$\{\height(y) : y \in \psi\iv(W_\xi)\}$.
Let $D$ be the diagonal intersection; so $D$ is club and
$\xi < \alpha \in D \to \alpha \in C_\xi$.  
Let $S$ be the set of limit $\alpha \in D$ such that
$\LL_\alpha(T) \cap \psi\iv(\bigcup \WW_n) \ne \emptyset$;
then $S$ is stationary.
For $\alpha \in S$,
choose $y_\alpha \in  \LL_\alpha(T) \cap \psi\iv(\bigcup \WW_n)$.
Then $y_\alpha \in \psi\iv(W_{\xi_\alpha})$ for some (unique) $\xi_\alpha$,
and $\xi_\alpha \ge \alpha$ since $\alpha\in D$. 
Then fix $x_\alpha \tro y_\alpha$
with $x_\alpha \up \cap y_\alpha \down \subseteq  \psi\iv(W_{\xi_\alpha})$.
By the Pressing Down Lemma, fix $x$ and a stationary $S' \subseteq S$ such
that $x_\alpha = x$ for all $\alpha\in S'$. 
Then, using $\xi_\alpha \ge \alpha$, fix stationary
$S'' \subseteq S'$ such that  the $\xi_\alpha$,
for $\alpha \in S''$, are all different.
Then the sets $x\up \cap y_\alpha \down$,
for $\alpha \in S''$ are pairwise disjoint,
which is impossible because  $\LL_{\height(x) + 1}(T)$ is countable.
\end{proof}

\begin{proofof}{Theorem \ref{thm-special-diamond}}
Call $\psi: T\to B$ a \emph{DP} map iff $\psi$ is continuous
and each $\psi\iv\{b\}$ is discrete.

We build $T$, along with an order--preserving $\varphi : T \to  \QQQ$,
and use $\diamondsuit$ to defeat all DP maps $\psi : T \to  B$.

As a set, $T$ will be the ordinal $\omega_1$, and the root will be $0$.
We shall define the tree order $\tro$ so that
$\LL_0(T) = \{0\}$,
$\LL_1(T) = \omega \backslash \{0\}$,
$\LL_{n+1}(T) = \{\omega \cdot n + k : k \in \omega\}$ for $0 < n < \omega$,
and $\LL_\alpha(T) = \{\omega \cdot \alpha + k : k \in \omega\}$ when
$\omega \le \alpha < \omega_1$.
As in the usual construction of a special Aronszajn tree, we construct
$\varphi: T \to  \QQQ$ and $\tro$ recursively so that $\varphi(0) = 0$ and
\[
\begin{array}{l}
\forall x \in T \, \forall \alpha < \omega_1 \, \forall q \in  \QQQ \, [
\alpha > \height(x) \wedge q > \varphi(x) \to \\
\qquad \exists y \in \LL_\alpha(T)\, [x \tro y \wedge \varphi(y) = q ]]
\ \ .
\end{array}
\tag{$\ast$}
\]
This implies, in particular,
that each node has $\aleph_0$ immediate successors.

Let $\langle \psi_\alpha  : \alpha < \omega_1 \rangle$
be a $\diamondsuit$ sequence, where each $\psi_\alpha : \alpha \to  B$.
Such a sequence exists by $\diamondsuit$ because $|B| \le  \aleph_1$.

In the recursive construction of $\tro$ and $\varphi$, do the usual thing in
building each $\LL_\gamma(T)$ to preserve $(\ast)$.  But in addition, whenever
$\omega \cdot \gamma = \gamma > 0$ (so $T_\gamma = \gamma$ as a set,
and $\psi_\gamma : T_\gamma \to  B$):
\emph{if} $\psi_\gamma$ is a DP map,
then \emph{if} it is possible, extend $\tro$ so that the node
$\gamma \in \LL_\gamma(T)$ satisfies:
\[
\sup\{\varphi(x) : x \tro \gamma\}  \le 1 
\text{ and }
\langle \psi_\gamma(x) : x \tro \gamma \rangle 
\text{ does not converge in } B \ \ .
\tag{$\dag$}
\]
This implies that $\psi_\gamma$ could not extend to a continuous
map into $B$.
Use the nodes $\gamma+1, \gamma + 2 ,\ldots$ to preserve $(\ast)$,
so if $(\dag)$ is possible, we may let $\varphi(\gamma) = 1$.
If $(\dag)$ is impossible, then ignore it and just preserve $(\ast)$. To ensure that the tree will be Hausdorff, make sure that if
$j\neq k$ then $\gamma + j$ and $\gamma+k$ are limits of distinct branches.

\begin{lemma}[Main Lemma] 
Suppose that $\psi : T \to  B$ is a DP map.  Then there is a 
club $C \subseteq \omega_1$ so that for all limit points $\gamma$ of $C$:
$\omega \cdot \gamma = \gamma$, and \emph{if} $\psi_\gamma = \psi \res\gamma$,
\emph{then} $(\dag)$ is possible at level $\gamma$.
\end{lemma}

The theorem follows immediately, since choosing such a $\gamma$
for which $\psi_\gamma = \psi \res\gamma$, we see that $\psi$
cannot be continuous at node $\gamma \in \LL_\gamma(T)$.

So, we proceed to prove the Main Lemma. 
We use a standard definition of $C$ --- namely,
let $\langle M_\xi : \xi < \omega_1 \rangle$ be a continuous 
chain of countable elementary submodels of $H(\theta)$ (for a suitably
large regular $\theta$), such that
$\varphi, \psi, \tro , B \,\in\, M_0$ and each $M_\xi \in M_{\xi + 1}$.
Let $C = \{M_\xi \cap \omega_1 : \xi < \omega_1\}$.

Now, fix a limit point $\gamma$ of $C$, with 
$\psi_\gamma = \psi \res\gamma$.  Let $\alpha_n \nearrow \gamma$,
with all $\alpha_n \in C$.
We shall build a Cantor tree of candidates for the
path satisfying $(\dag)$, and then prove that one of these
works by using the fact that $B$ does not have a Cantor subset.
For $s \in 2^{< \omega}$, construct $W_s, U_s, x_s$ with the following
properties; here, $|s|$ denotes the length of $s$.

\begin{itemizz}
\item[1.] $W_s \subseteq B$ is open and non-empty,
and $\diam(W_s) \le 1 / |s|$.
\item[2.]  $W_\emptyset = B$.
\item[3.] $\overline{W_{s \cat 0}} , \overline{W_{s \cat 1}} \subseteq W_s$ and
$\overline{W_{s \cat 0}} \cap \overline{W_{s \cat 1}} = \emptyset $.
\item[4.] $U_s$ is a stationary open subset of $T$, with $(U_s)^p = U_s$.
\item[5.] $U_\emptyset = \{x \in T : \varphi(x) < 1\}$,
\item[6.] ${U_{s \cat 0}} , {U_{s \cat 1}} \subseteq U_s$ and
$U_s \subseteq \psi\iv( W_s )$.
\item[7.]
$x_s \in U_s$ and $U_{s\cat i} \subseteq x_s \up$ for $i = 0,1$.
\item[8.]
$x_\emptyset = 0$, the root node of $T$.
\item[9.] For $n = |s|$:
$\height(x_s) < \alpha_n$ and, when $n > 0$,
$\height(x_s) \ge \alpha_{n-1}$.
\item[10.] For $n = |s|$ and $\alpha_n = M_{\xi_n} \cap \omega_1$:
$W_s, U_s, x_s \in M_{\xi_n}$.
\end{itemizz}
For each $f \in 2^\omega$, conditions (7) and (9) guarantee that
$P_f := \bigcup \{ x_{f\res n} \down : n \in \omega\}$
is a cofinal path through $T_\gamma$.
Now, fix $f$ so that $\bigcap_{n \in \omega} W_{f\res n} = \emptyset$.
There is such an $f$ because otherwise, by conditions (1)(3),
$\bigcup \{ \bigcap_{n \in \omega} W_{f\res n} : f \in 2^\omega\}$
would be a Cantor subset of $B$.
Then, $(\dag)$ will hold if we place node $\gamma$ above
the path $P_f$; note that condition (5) guarantees that
$\sup\{\varphi(x) : x \tro \gamma\}  \le 1$, 
and every limit point of $\langle \psi_\gamma(x) : x \tro \gamma \rangle $
must lie in $\bigcap_{n \in \omega} W_{f\res n}$,
which is empty.

Of course, we need to verify that the $W_s, U_s, x_s$ can
be constructed.  Fix $s$, with $n = |s|$,
and assume that we have $W_s, U_s, x_s$.
Note that $ U_s \cap x_s \up$ is stationary by $(U_s)^p = U_s$.
Applying Lemma \ref{lemma-parac}
\big(to $\psi \res (U_s \cap x_s \up) : (U_s \cap x_s \up) \to W_s$\big),
there exist $b_0 \ne b_1$ in $W_s$ such that 
$\psi\iv( N_\varepsilon(b_i) ) \cap U_s \cap x_s \up$ is stationary
for all $\varepsilon > 0 $; applying 
condition (10), choose such $b_0, b_1 \in M_{\xi_n}$.
Then fix $\varepsilon$ to be the smallest of $1/(n+1)$,
$d(b_0, b_1) / 3$, 
$d(b_0, B\backslash W_s)   / 2$, and $ d(b_1, B\backslash W_s) / 2$.
Let $W_{s\cat i} = N_\varepsilon(b_i)$ and 
$U_{s\cat i} = ( \psi\iv( W_{s\cat i} ) \cap U_s \cap x_s \up) ^ p$.

Then choose  $x_{s\cat i} \in U_{s\cat i}$ with
$\alpha_n \le \height( x_{s\cat i} )$; 
such an $x_{s\cat i} $ exists by $(U_{s\cat i})^p = U_{s\cat i}$.
Also, make sure that $x_{s\cat i} \in M_{\xi_{n+1}}$
(using $M_{\xi_{n+1}} \prec H(\theta)$), 
which guarantees that $\height( x_{s\cat i} ) < \alpha_{n + 1}$
and that condition (10) will continue to hold.
\end{proofof}

\section{Constructing Continuous Maps}
\label{sec-const}

\begin{proofof}{Theorem \ref{thm-build-map}}
Let $H = \{1,4,16, \ldots\} = \{2^{2i} : i \in \omega\}$ and
$K = \{2,8,32, \ldots\} = \{2^{2i + 1} : i \in \omega\}$.
Observe that $H \cap K = \emptyset$ and
\[
\forall n_1, n_2 \in H\, \forall j_1, j_2 \in K\, 
 [ n_1 + j_1 = n_2 + j_2 \; \to\;  n_1 = n_2 \wedge j_1 = j_2 ] \ \ .
\]
Now, let $P$ be the set of all real numbers of the form
$\sum_{j \in K} \varepsilon_j 2^{-j}$, where each 
$\varepsilon_j \in \{0,1\}$.  Then $P$ is a Cantor set
and $0 \in P \subset [0,1]$.

Let $S$ be the set of all sums of the form
$\sum_{n \in H} z_n 2^{-n}$, where each $z_n \in P$.
Then $S$ is compact, since it is the range of the continuous
map $\Gamma : P^H \to \RRR$ defined by
$\Gamma (\vec z) = \sum_{n \in H} z_n 2^{-n}$.  Also, $\Gamma$ is 1-1;
that is,
\[
\sum_{n \in H} z_n 2^{-n} = \sum_{n \in H} w_n 2^{-n} 
\;\Rightarrow\;
\forall n \in H \, [z_n = w_n] \qquad (\text{all $z_n, w_n \in P$ } )\ \ .
\tag{\ding{"60}}
\]
To see this, let $z_n = \sum_{j \in K} \varepsilon_{j,n} 2^{-j}$ and
$w_n = \sum_{j \in K} \delta_{j,n} 2^{-j}$.  We then have
$\sum\{ \varepsilon_{j,n} 2^{-(j + n) } : j \in K\wedge n \in H \}   =
\sum\{ \delta_{j,n} 2^{-(j + n) } : j \in K\wedge n \in H \}   $.
Since the values $j + n$ are all different, each 
$\varepsilon_{j,n} = \delta_{j,n}$.

For $n \in H$, define the ``coordinate projection'' $\pi_n: S \to P$ so that
we have
$\pi_n ( \sum_{n \in H} z_n 2^{-n}  ) = z_n$.  So,
$\pi_n = \hat\pi_n \circ \Gamma\iv$, where $\hat\pi_n : P^H \to P$
is the usual coordinate projection.

Since $T$ is special, fix $a : T \to H$ such that each
$A_n :=  a\iv\{n\}$ is antichain.
Also, fix a 1-1 function $\zeta : T \to P \backslash \{0\}$ such
that $\zeta(T)$ has no perfect subsets.  Then, define
\[
\psi(x) = \sum \{ \zeta(t) \cdot 2^{- a(t)} : t \in x\down\} \ \ .
\]
Let $B$ be the range of $\psi$;  then $\psi : T \to B$ is clearly
continuous and order preserving.

Note that $\psi(x) = \sum_{n \in H} z_n 2^{-n} $,
where $z_n = \zeta(t)$ if $t \in A_n \cap x \down$,
and $z_n = 0$ if $A_n \cap x \down = \emptyset$.  Then,
$x\down \ne y\down \to \psi(x) \ne \psi(y)$
follows from (\ding{"60}) and the fact that $\zeta$ is 1-1.

Suppose that $C \subseteq B$ is a Cantor set.
Then each $\pi_n(C)$ is a compact subset of $\ran(\zeta) \cup \{0\}$,
and is hence countable.  There is then a countable $\alpha$
such that $\pi_n(C) \subseteq \zeta(T_\alpha) \cup \{0\}$ for all $n \in H$.
So, fix $x \in T$ with $\psi(x) \in C$ and $\height(x) > \alpha$,
let $x\down \cap \LL_\alpha(T) = \{t\}$, and let $n = a(t)$.
Then $\zeta(t) = \pi_n(\psi(x)) \in \pi_n(C)$ and 
$\zeta(t) \notin \zeta(T_\alpha) \cup \{0\}$, a contradiction.
\end{proofof}

\end{document}